\documentclass[11pt]{amsart}
\usepackage{amsfonts,amssymb,amsthm}
\usepackage{amsmath,amscd}
\usepackage{pstricks}
\usepackage{pstricks,pst-node}
\usepackage{mathrsfs}
\usepackage[all]{xy}

\renewcommand{\subsection}[1]{\vspace{.18in}\par\noindent\addtocounter{subsection}{1}
\setcounter{equation}{0}{\bf\thesubsection.\hspace{5pt}#1}}

\theoremstyle{definition}
\newtheorem{Rem}[subsection]{Remark}

\theoremstyle{plain}
\newtheorem{Prop}[subsection]{Proposition}
\newtheorem{Thm}[subsection]{Theorem}
\newtheorem{Lem}[subsection]{Lemma}
\newtheorem{Coro}[subsection]{Corollary}
\numberwithin{equation}{subsection}

\newcommand{\spann}{\mathrm{span}}

\newcommand{\sk}{s_k}
\newcommand{\tuk}{{u}_k}

\newcommand{\han}{\subseteq}
\def\leq{\leqslant}\def\geq{\geqslant}
\newcommand{\hl}{\widehat L}
\newcommand{\soc}{\mathrm{soc}}

\newcommand{\dt}{\delta}
\newcommand{\Dt}{\Delta}

\newcommand{\Og}{\Omega}
\newcommand{\og}{\omega}

\newcommand{\up}{\upsilon}
\newcommand{\ep}{\varepsilon}
\newcommand{\al}{\alpha}
\newcommand{\bt}{\beta}
\newcommand{\ga}{\gamma}
\newcommand{\h}{\widehat}
\newcommand{\ti}{\tilde}

\newcommand{\zr}{\zeta_r}

\newcommand{\ot}{\otimes}

\newcommand{\bfOg}{{\bf\Omega}}

\newcommand{\sZ}{\mathcal Z}

\newcommand{\ba}{\bar}
\newcommand{\ol}{\overline}

\newcommand{\ra}{\rightarrow}
\newcommand{\la}{\lambda}
\newcommand{\La}{\Lambda}

\newcommand{\lb}{{\overline{\lambda}}}
\newcommand{\mb}{{\overline{\mu}}}

\newcommand{\mbn}{\mathbb N}
\newcommand{\mbq}{\mathbb Q}
\newcommand{\mbz}{\mathbb Z}

\newcommand{\bfU}{\mathbf{U}}

\newcommand{\Ga}{\Gamma}

\newcommand{\rad}{\operatorname{rad}}

\newcommand{\End}{\operatorname{End}}
\newcommand{\Hom}{\operatorname{Hom}}

\newcommand{\Ogk}{\Og_k}

\newcommand{\sI}{\boldsymbol{\mathcal I}}
\newcommand{\sJ}{\boldsymbol{\mathcal J}}
\newcommand{\sB}{\frak{B}}
\newcommand{\hsB}{\widehat{\frak B}}
\newcommand{\St}{\mathrm {St}}

\begin{document}
\title{On the structure of $\End_{u_k(2)}(\Ogk^{\otimes r})$}

\author{Qiang Fu}
\address{Department of Mathematics, Tongji University, Shanghai, 200092, China.}
\email{q.fu@hotmail.com}
\author{Qunguang Yang}
\address{Department of Mathematics, Tongji University, Shanghai, 200092, China.}
\email{ynmyff@163.com}
\date{\today}

\thanks{Supported by the National Natural Science Foundation
of China(10971154) and the Program NCET}

\sloppy \maketitle

\begin{abstract}
Let $u_k(2)$ be the infinitesimal quantum $\frak{gl}_2$ over $k$, where $k$ is a field
containing an $l$th primitive root $\ep$ of 1 with $l\geq 3$ {\it
odd}. We will determine the basic algebra for $\End_{u_k(2)}(\Ogk^{\otimes r})$, where $\Ogk$ is the natural module for $u_k(2)$.
\end{abstract}

\section{Introduction}

Infinitesimal quantum groups (i.e., small quantum groups) are important finite dimensional Hopf algebras introduced  by G. Lusztig in \cite{Lu90}. They are quantum versions of restricted enveloping algebras in modular Lie theory and they are related to many other mathematical object. For example, certain important relation between infinitesimal quantum groups and logarithmic conformal field theories has been found in \cite{FGST}.

Let $k$ be a field
containing an $l$th primitive root $\ep$ of 1 with $l\geq 3$ {\it
odd}.
Let $u_k(2)$ be the infinitesimal quantum $\frak{gl}_2$ over $k$ and let $\Og_k$ be the natural module for $u_k(2)$. Since $u_k(2)$ is a Hopf algebra, $\Og_k^{\ot r}$ becomes a $u_k(2)$-module. The module $\Og_k^{\ot r}$ is very important since it contains many information for $u_k(2)$. We are interested in the algebra  $\End_{u_k(2)}(\Ogk^{\ot r})$. In this paper, we will study the basic algebra for $\End_{u_k(2)}(\Ogk^{\otimes r})$. We expect that our results can be related to some sort of Schur--Weyl duality for infinitesimal quantum groups.

In \cite{DFW1,Fu07}, the little $q$-Schur algebra $u_k(n,r)$ was introduced as a homomorphic image of the infinitesimal quantum group $u_k(n)$ and the symmetry structure for $u_k(n,r)$ was investigated
through the construction of various bases of monomial, BLM and PBW types for $u_k(n,r)$. Representation theory for $u_k(n,r)$ was studied in \cite{DFW2}. Little $q$-Schur algebras are closely related to infinitesimal $q$-Schur algebras introduced in  \cite{DNP961,Cox97,Cox00} (see \cite{Fu05}) and they are useful in the investigation of $\End_{u_k(2)}(\Ogk^{\otimes r})$ since $\End_{u_k(2)}(\Ogk^{\otimes r})=\End_{u_k(2,r)}(\Ogk^{\otimes r})$.

We organize this paper as follows. We will recall the definition of little $q$-Schur algebras in \S2. Using the result of \cite{CP}, the basic algebra of the infinitesimal quantum $\frak{sl}_2$ will be studied in \S3. In \S4, we will classify
semisimple blocks for the infinitesimal $q$-Schur algebra $s_k(2,r)$ and the little $q$-Schur algebra $u_k(2,r)$.
Finally, we will  determine the basic algebra for $\End_{u_k(2)}(\Ogk^{\ot r})$ in \S5.


\section{The little $q$-Schur algebra}

The quantum enveloping algebra of $\frak{gl}_2$ is the algebra
$\bfU(2)$ over $\mathbb Q(\up)$ (with $\up$ an indeterminate) presented by generators
$$E,\ F,\ K_1^{\pm 1},\ K_2^{\pm 1}$$
and relations
\begin{itemize}
\item[(a)] $ K_{1}K_{2}=K_{2}K_{1},\ K_{1}K_{1}^{-1}=1,\ K_{2}K_{2}^{-1}=1;$

\item[(b)] $ K_{1}E=\up EK_{1},\ K_{2}E=\up^{-1}EK_{2};$

\item[(c)] $ K_{1}F=\up^{-1} FK_{1}\ K_{2}F=\up FK_{2};$

\item[(d)] $ EF-FE=\frac
{K-K^{-1}}{\up-\up^{-1}},\
where \  {K} =K_{1}K_{2}^{-1}.$
\end{itemize}
The algebra $\bfU(2)$ is a Hopf algebra with comultiplication
$\Dt$ defined on generators by $ \Dt(E)=E\ot K+1\ot
E$, $\Dt(F)=F\ot 1+K^{-1}\ot F$, $\Dt(K_i)=K_i\ot
K_i$. Let $\sZ=\mbz[\up,\up^{-1}]$.
Following \cite{Lu90}, let $U_\sZ(2)$  be the $\sZ$-subalgebra of $\bfU(2)$ generated by all $E^{(m)}$, $F^{(m)}$, $K_i^{\pm 1}$ and
$\left[ {K_i;0 \atop t} \right]$, where for $m,t\in\mathbb N$,
$$E^{(m)}=\frac{E^m}{[m]^!},\,\,F^{(m)}=\frac{F^m}{[m]^!},\text{ and }
\bigg[ {K_i;0 \atop t} \bigg] =
\prod_{s=1}^t \frac
{K_i\up^{-s+1}-K_i^{-1}\up^{s-1}}{\up^s-\up^{-s}}$$
with $[m]^{!}=[1][2]\cdots[m]$ and
$[i]=\frac{\up^i-\up^{-i}}{\up-\up^{-1}}$.

Let $k$ be a field containing an $l$th primitive root $\ep$ of 1
with $l\geq 3$ odd. Specializing
$\up$ to $\ep$, $k$ will be viewed as an $\sZ$-module.
Let $U_k(2)=U_\sZ(2)\otimes_\sZ k$. We will denote the
images of $E$, $F$, $K_i$ in $U_k(2)$ by the same letters.
Let $\ti u_k(2)$ be the $k$-subalgebra of $U_k(2)$ generated by the
elements $E$, $F$, $K_i^{\pm 1}$ for all $i$. The algebra
$u_k(2)=\ti u_k(2)/\langle K_1^l-1,K_2^l-1\rangle$ is the infinitesimal quantum group of $\frak{gl}_2$ (cf. \cite{Lu90}).

Let $\Og$ be a free $\sZ$-module with basis
$\{\og_i\mid 1\leq i\leq 2\}$ and let $\bfOg=\Og\otimes\mbq(\up)$.
Then $\bfU(2)$ acts naturally on $\bfOg$ by
$K_a\og_b=\up^{\dt_{a,b}}\og_b$,
$E\og_b=\dt_{2,b}\og_{b-1}$ and
$F\og_b=\dt_{1,b}\og_{b+1}$. Since $U_\sZ(2)\Og\han\Og$, $\Ogk:=\Og\ot k$ becomes a $U_k(2)$-module. Thus the
tensor space $\Ogk^{\ot r}$ is an $U_k(2)$-module via the
comultiplication $\Dt$ on $U_k(2)$. Consequently, we get an algebra homomorphism
$\zeta_r:U_k(2)\ra\End(\Ogk^{\ot r})$.
Then $S_k(2,r):=\zeta_r(U_k(2))$ is the $q$-Schur algebra over $k$ (see \cite{Du95}). The algebra $\tuk(2,r)=\zeta_r(\ti u_k(2))$ is called a little $q$-Schur algebra (cf. \cite{DFW1,Fu07}).
Since  $\zeta_r(K_i^l-1)=0$ for
all $i$, $\zeta_r$ induces a surjective map
$$\zeta_r:u_k(2)\twoheadrightarrow \tuk(2,r).$$
Thus we have $\End_{u_k(2)}(\Og_k^{\ot r})=\End_{u_k(2,r)}(\Og_k^{\ot r}).$

Let $s_k(2,r)$ be the subalgebra of $S_k(2,r)$ generated by $\tuk(2,r)$ and
$\zeta_r\big(\big[ {K_i;0 \atop t} \big]\big)$ for $1\leq i\leq 2$ and $t\in\mbn$. By \cite{Fu05}, $s_k(2,r)$ is isomorphic to the infinitesimal $q$-Schur algebra introduced in \cite{DNP961,Cox97,Cox00}.
The representation of little $q$-Schur algebras and infinitesimal $q$-Schur algebras has close relation. For example we have the following result.
\begin{Lem}[{\cite[7.1 and 8.5]{DFW2}}]\label{socle-indc}
Let $V$ be a $s_k(2,r)$-module. Then
$\soc_{s_k(2,r)}V=\soc_{\tuk(2,r)}V$. Furthermore $V$ is an indecomposable
$s_k(2,r)$-module if and only if $V$ is an indecomposable
$\tuk(2,r)$-module.
\end{Lem}

\section{The basic algebra of $u_k'(2)$}
Let $u_k'(2)$ be the subalgebra $u_k(2)$ generated by $E$, $F$, $K^{\pm 1}$. Then $u_k'(2)$ is the infinitesimal quantum group of $\frak{sl}_2$. In this section, we will investigate the basic algebra of $u_k'(2)$ using the result of \cite{CP}. This result will be used in \S4.

By \cite[9.4]{DFW2}, we have
\begin{equation}\label{image of sl2}
u_k(2,r)=\zr(u_k'(2))
\end{equation}
 for any $r$. Thus we have
$$ \End_{u_k(2)}(\Ogk^{\otimes r})
=\End_{u_k(2,r)}(\Ogk^{\otimes r})
= \End_{u_k'(2)}(\Ogk^{\otimes r}).$$
Now following \cite{CP}, we introduce several $u_k'(2)$-modules as follows. For $1\leq j\leq l$, let $L_j$ be the $u_k'(2)$-module with basis $v_0,v_1,\cdots,v_{j-1}$
such that the action of $u_k'(2)$ is given by
\begin{equation*}
Kv_i =\ep^{j-1-2i}v_i,\quad
Ev_i =[j-i]v_{i-1},\quad
Fv_i =[i+1]v_{i+1},
\end{equation*}
where $v_{-1}=v_j=0$. By \cite[3.2]{CP}, we know that the $u_k'(2)$-modules  $L_j$ $(1\leq j\leq l)$ are all non-isomorphic irreducible $u_k'(2)$-modules. For $1\leq j\leq l-1$, let $P_j$ be the $u_k'(2)$-module with basis $v_0^{(j)},v_1^{(j)},\cdots,v_{l-1}^{(j)},w_0^{(j)},w_{1}^{(j)},\cdots,w_{l-1}^{(j)}$ such that the action of $u_k'(2)$ is given by
\begin{equation*}
\begin{split}
&\qquad\qquad\qquad
Kv_s^{(j)} =\ep^{l-j-2s-1}v_s^{(j)},\quad
Ev_s^{(j)}=[l-j-s]v_{s-1}^{(j)},\quad
Fv_s^{(j)}=[s+1]v_{s+1}^{(j)},\\
&Kw_s^{(j)} =\ep^{j-2s-1}w_s^{(j)},\quad
Ew_s^{(j)}=[j-s]w_{s-1}^{(j)}+\bigg[{l-j-1+s\atop s}\bigg]v_{s-j-1+l}^{(j)},\quad
Fw_s^{(j)} =[s+1]w_{s+1}^{(j)},
\end{split}
\end{equation*}
where $v_s=w_s=0$ for $s\not\in[0,l-1]$. By \cite[3.9]{CP}, the module $P_j$ is the projective cover of $L_j$ for $1\leq j\leq l-1$.
For $1\leq j\leq l-1$, let
\begin{equation*}
\begin{split}
V_j&=\spann_k\{v_{l-j}^{(j)},\cdots,v_{l-1}^{(j)}\},\\
M_j&=\spann_k\{v_{0}^{(j)},\cdots,v_{l-1}^{(j)}\},\\
N_j&=\spann_k\{v_{l-j}^{(j)},\cdots,v_{l-1}^{(j)},
w_j^{(j)},\cdots,w_{l-1}^{(j)}\}.
\end{split}
\end{equation*}
It is easy to check the following result.
\begin{Lem}\label{Pj}
For $1\leq j\leq l-1$, $V_j$, $M_j$ and $N_j$ are all submodules of $P_j$. Furthermore we have  $V_j\cong L_j$, $M_j+N_j=\rad(P_j)$, $M_j/L_j\cong N_j/L_j\cong L_{l-j}$. Thus
$P_j$ has the following structure:
\begin{eqnarray*}
P_j: \quad & L_j      \\
    &2 L_{l-j}      \\
  &L_j
\end{eqnarray*}

\end{Lem}

For $1\leq j\leq l-1$, let $\al_1^{(j)}:P_j\ra P_{l-j}$ be the linear map such that
\begin{equation*}
\al_1^{(j)}(M_j)=0,\ \al_1^{(j)}(w_s^{(j)})=v_s^{(l-j)}
\end{equation*}
for $0\leq s\leq l-1$,
and let $\al_2^{(j)}:P_j\ra P_{l-j}$ be the linear map such that
\begin{equation*}
\al_2^{(j)}(N_j)=0,\ \al_2^{(j)}(v_s^{(j)}) =a_s^{(j)}v_{s+j}^{(l-j)},\ \al_2^{(j)}(w_t^{(j)})=b_t^{(j)}w_{t+l-j}^{(l-j)}
\end{equation*}
for $0\leq s\leq l-j-1$ and $0\leq t\leq j-1$,
where $a_s^{(j)}=[j]\big[{s+j\atop s}\big]$ and $b_t^{(j)}=[l-j]\big[{l+t-j\atop t}\big]$.
One can easily check the following result.
\begin{Lem}\label{basis}
For $1\leq j\leq l-1$, we have $\al_1^{(j)},\al_2^{(j)}\in\Hom_{u_k'(2)}(P_j,P_{l-j})$ and the set $\{\al_1^{(j)},\al_2^{(j)}\}$ forms a $k$-basis for $\Hom_{u_k'(2)}(P_j,P_{l-j})$.
\end{Lem}

Let $\mathscr Q$ be the quiver in Figure 1
\begin{figure}[h]
\begin{center}
\setlength{\unitlength}{.5cm}
\begin{picture}(18,3)
 \put(6.9,1.9){$\bullet$} \put(6.4,1.2){$X_0$}
\put(10.9,1.9){$\bullet$} \put(10.7,1.2){$X_1$}
  \put(7.4,2.2){\vector(1,0){3.2}}
   \put(7.4,2.9){\vector(1,0){3.2}}
\put(9,3.1){$\alpha_{1}$}  \put(9,2.4){$\alpha_{2}$}
\put(10.6,1.8){\vector(-1,0){3.2}} \put(9,1.1){$\beta_{1}$}
\put(10.6,0.9){\vector(-1,0){3.2}} \put(9,0.2){$\beta_{2}$}
\end{picture}
\caption{}
\end{center}
\end{figure}
and $\mathscr I$ be the two sided ideal of the path algebra $k\mathscr Q$ generated by the relations
$\al_1\bt_2=\al_2\bt_1=0$,  $\al_1\bt_1=\al_2\bt_2$,  $\bt_1\al_2=\bt_2\al_1=0$,  $\bt_1\al_1=\bt_2\al_2$.
Let $\mathscr A=k\mathscr Q/\mathscr I$,

\begin{Lem}\label{basic algebra of sl2}
For $1\leq j\leq l-1$, let $B_j$ be the basic algebra of the block of $u_k'(2)$ containing $L_j$. Then $B_j$ has two simple modules $L_j,L_{l-j}$ and $B_j\cong\mathscr A$.
\end{Lem}
\begin{proof}
Since $P_j$ is the projective cover of $L_j$ we have $B_j=\End_{u_k'(2)}(P_j\oplus P_{l-j})$. It is easy to see that there is an algebra homomorphism $f:\mathscr A\ra B_j$ such that $f(\al_1)=\al_1^{(j)}$, $f(\al_2)=\al_2^{(j)}$, $f(\bt_1)=\al_2^{(l-j)}$ and $f(\bt_2)=\al_1^{(l-j)}$. By \ref{basis}, we conclude that $f$ is surjective. Thus, since $\dim_k(\mathscr A)=\dim_k B_j=8$, $f$ must be an algebra isomorphism.
\end{proof}

\section{The semisimple block of $s_k(2,r)$ and $\tuk(2,r)$}
In this section, we will first classify semisimple blocks for $s_k(2,r)$ and $\tuk(2,r)$. Certain relation between indecomposable projective modules for $s_k(2,r)$ and $\tuk(2,r)$ will be given in \ref{restriction of projective module}. This result can be used to determine the basic algebra of non-semisimple blocks for $\tuk(2,r)$ in certain cases (see \ref{basic algebra}).

Following \cite[Section 3.1 and 3.2]{Donbk}, let $G_1$ be the first
Frobenius kernel of quantum linear group  and $G_1T$ be the corresponding Jantzen
subgroups. By
\cite[3.1(13)(i)]{Donbk}, for each $\la\in\mbz^2$, there is a
simple object $L_1(\la)$ in Mod($G_1$) and a simple object $\h
L_1(\la)$ in Mod$(G_1T)$. Note that if $\lb=\bar\mu$, then $  L_1(\la)\cong  L_1(\mu)$.  Here $\ba\ :\mbz^2\ra(\mathbb Z/l\mathbb Z)^2$ is the map defined by
$\ol{(j_1,j_2)}=(\ol{j_1},\ol{j_2}).$
Thus we may denote $L_1(\la)$ by $L_1(\lb)$.

Let  $P(D)=\mathbb{N}^2$  and $P_1(D)=\{\lambda\in \mbn^2|0\leqslant \lambda_1-\lambda_{2}\leqslant l-1, 0\leqslant \lambda_2\leqslant l-1\}$. Then by \cite{Cox97,Cox00}, the set $\{\widehat{L}_1(\lambda)|\lambda\in \Gamma_1^r(D)\}$ forms a complete set of non-isomorphic simple $s_k(2,r)$-modules, where  $\Gamma_1^r(D)=\{\lambda\in P_1(D)+lP(D)|\la_1+\la_2=r\}$. Furthermore, by \cite[5.5]{DFW2}, the set \{$L_1(\lb)|\bar{\lambda}\in \overline{\Gamma_1^r(D)}$\} forms a complete set of non-isomorphic simple $\tuk(2,r)$-modules.

We will  denote the
block of infinitesimal $q$-Schur algebras $s_k(2,r)$ containing
$\hl_1(\la)$ by $\hsB_1^r(\la)$ for $\la\in\Ga_1^r(D)$. For
$\bar\la\in\ol{\Ga_1^r(D)}$, the block of little $q$-Schur algebras $\tuk(2,r)$
containing $L_1(\lb)$ will be denoted by $\sB_1^r(\bar\la)$.

For $\la\in\mbz^2$, we write $\la_1-\la_2+1=ml+s$, where $0\leq
s<l$. Let $s^-\cdot\la=(\la_2-1+ml,\la_1+1-ml)$. Similarly, let
$s^+\cdot\la=\la$ if $s=0$, otherwise let
$s^+\cdot\la=(\la_2-1+(m+1)l,\la_1+1-(m+1)l)$. Define $s^0\la=\la$,
$s^m\cdot\la=s^+\cdot(s^{m-1}\cdot\la)$ for $m>0$, and
$s^m\cdot\la=s^-\cdot(s^{m+1}\cdot\la)$ for $m<0$.

To study the block for $\tuk(2,r)$, we need to understand the block for $s_k(2,r)$.
Let $$\text{St}=\{\la\in\Ga_1^r(D)\mid \la_1-\la_2+1=ml\text{ for some $m\in\mbz$}\}.$$  An element in $\text{St}$ will be called a Steinberg weight. For $\la\in\St$, the block $\h \sB_1^r(\la)$ of $s_k(2,r)$ will be called a Steinberg block. Note that $\h \sB_1^r(\la)$ is semisimple for $\la\in\text{St}$. The set $\text{St}$ can be explicitly  calculated as follows.

\begin{Lem}\label{St}
Assume $r=sl+t$, where $s\geq 0$ and $0\leq t<l$.

(1) If $s\geq 1$ and $t$ is odd, then
$$\St=\bigg\{\bigg(l+\frac{t-1}{2},\frac{t+1}{2}\bigg)+
l\bigg(\frac{s+m-2}{2},\frac{s-m}{2}\bigg)
\,\bigg| m=s,s-2,\cdots,-s+2\bigg\}.$$

(2) If $s\geq 1$, $t$ is even and $t\not=l-1$, then
$$\St=\bigg\{\bigg(\frac{3l-1+t}{2},\frac{t+l+1}{2}\bigg)+
l\bigg(\frac{s+m-3}{2},\frac{s-m-1}{2}\bigg)
\,\bigg| m=s-1,s-3,\cdots,-s+3\bigg\}.$$

(3) If $s\geq 1$ and $t=l-1$, then
$$\St=\bigg\{(l-1,0)+
l\bigg(\frac{s+m-1}{2},\frac{s-m+1}{2}\bigg)
\,\bigg| m=s+1,s-1,\cdots,-s+1\bigg\}.$$

(4) If $s=0$, then
$$\St=
\begin{cases}
(l-1,0)&\text{if $r=l-1$},\\
\emptyset&\text{otherwise}.
\end{cases}$$
\end{Lem}
\begin{proof}
We only prove (1). The other cases are proved similarly.
We assume $t$ is odd. Clearly, we have  $\big(l+\frac{t-1}{2},\frac{t+1}{2}\big)+
l\big(\frac{s+m-2}{2},\frac{s-m}{2}\big)\in\St$ for $m=s,s-2,\cdots,-s+2$. If $\la\in\St$ then $\la_1-\la_2+1=lm$ for some $m\in\mbz$. Hence $\la_1=\frac{l(s+m)+t-1}{2}$ and $\la_2=\frac{l(s-m)+t+1}{2}$.
We write $$\la=\bigg(l+\frac{t-1}{2},\frac{t+1}{2}\bigg)+
l\bigg(\frac{s+m-2}{2},\frac{s-m}{2}\bigg).$$
Since $\big(l+\frac{t-1}{2},\frac{t+1}{2}\big)\in P_1(D)$ and $\la\in\Ga_1^r(D)$ we conclude that $\big(\frac{s+m-2}{2},\frac{s-m}{2}\big)\in\mbn^2$ and hence
$-s+2\leq m\leq s$ and $m-s$ is even. The assertion (1) follows.
\end{proof}

The following result follows from \cite[3.4]{EF} (see also \cite[2.3]{DNP962}).
\begin{Lem}\label{block1}
For $\la\in\Ga_1^r(D)$ we have
$$\h \sB_1^r(\la)=\{s^m\cdot\la\mid m\in\mbz,\,s^m\cdot\la\in\Ga_1^r(D)\}.$$
\end{Lem}

By  \ref{block1}, we immediately get the following result.
\begin{Coro}\label{block2}
Assume $\la\not\in\St$ and  $\la\in\Ga_1^r(D)$. Write
$\la=\mu+l\nu$, where $\mu\in P_1(D)$ and $\nu\in P(D)$.

(1) If $\mu_1\leq l-2$, then
$\h\sB_1^r(\la)=\{s^m\cdot\la\mid-2\nu_1\leq m\leq 2\nu_2\}$.

(2) If $\mu_1\geq l-1$, then
$\h\sB_1^r(\la)=\{s^m\cdot\la\mid-2\nu_1-1\leq m\leq 2\nu_2+1\}$.
\end{Coro}

Now let us classify the semisimple block for $s_k(2,r)$ as follows.
Let
\begin{equation*}
\begin{split}
\sI &=\left\{\bigg(\frac{r+i}{2},\frac{r-i}{2}\bigg)\,\big|\,
0 \leq i\leq l-2;\,i<2l-r-2;\,r-i\geq 0\text{ is even}\right\};\\
\sJ &=\left\{\bigg(\frac{r+i}{2},\frac{r-i}{2}\bigg)\,\big|\,
0 \leq i\leq l-2;\,i\geq 2l-r-2;\,r-i\text{ is even}\right\}.
\end{split}
\end{equation*}

\begin{Prop}\label{semisimple non-Steinberg block}
The blocks $\h \sB_1^r(\la)$ $(\la\in\sI)$ are all semisimple
non-Steinberg blocks of $s_k(2,r)$.
\end{Prop}
\begin{proof}
Assume $\h\sB_1^r(\la)$ ($\la\in\Ga_1^r(D)$) is a semisimple
non-Steinberg block of $s_k(2,r)$. We need to prove that
$\la\in\sI$. We write $\la=\mu+l\nu$, where $\mu\in P_1(D)$ and
$\nu\in P(D)$. Since $\h\sB_1^r(\la)$ is semisimple, we have
$|\h\sB_1^r(\la)|=1$. So by \ref{block2} we have $\mu_1\leq l-2$ and
$\nu=0$. We denote $i=\la_1-\la_2$. Then
$\la=(\frac{r+i}{2},\frac{r-i}{2})$ and $r-i=2\la_2\geq 0$ is even.
Furthermore, since $\la=\mu\in P_1(D)$ and $\la_1=\mu_1\leq l-2$, we
have $0\leq i\leq\la_1\leq l-2$ and $2l-r-2-i=2l-2-2\la_1\geq
2l-2-2(l-2)>0$. Thus $\la\in\sI$.

On the other hand, by \cite[6.0.2]{EF}, $\h \sB_1^r(\la)$ is
semisimple for $\la\in\sI$. Furthermore, if
$\la=(\frac{r+i}{2},\frac{r-i}{2})\in\sI$, then we have $r+i< 2l-2$
and $r+i$ is even. So $r+i\leq 2l-4$. It follows that
$\la_1-\la_2+1\leq\la_1+1=\frac{r+i}{2}+1\leq l-1$ and hence
$\la\not\in\St$. This completes the proof.
\end{proof}

By \cite{EN}, we know that $s_k(2,r)=S_k(2,r)$ is semisimple for
$l>r$. If $l\leq r$, we can determine all non-semisimple blocks of
$s_k(2,r)$ as follows.

\begin{Prop}\label{non-semisimple block}
Assume $l\leq r$. Then the blocks $\h \sB_1^r(\la)$ $(\la\in\sJ)$ are all non-semisimple blocks of $s_k(2,r)$.
\end{Prop}
\begin{proof}
By \cite[6.0.2]{EF}, $\h \sB_1^r(\la)$ is non-semisimple for $\la\in\sJ$. Now we assume $\sB$ is a non-semisimple block of $s_k(2,r)$. We need to
prove that $\sB\cap\sJ\not=\emptyset$. Fix $\la\in\sB$. We write
$\la=\mu+l\nu$ with $\mu\in P_1(D)$ and $\nu\in P(D)$. Let
$\dt=s^{\nu_2-\nu_1}\cdot\la$. Then by \ref{block2} we have
$\dt\in\sB$. So it is enough to prove that $\dt\in\sJ$. By
definition we have
\begin{equation*}
\dt=
\begin{cases}
(\mu_2-1,\mu_1+1)+l(\frac{\nu_1+\nu_2+1}{2},\frac{\nu_1+\nu_2-1}{2})&\text{if $\nu_2-\nu_1$ is odd;}\\
\mu+l(\frac{\nu_1+\nu_2}{2},\frac{\nu_1+\nu_2}{2})&\text{if $\nu_2-\nu_1$ is even}.
\end{cases}
\end{equation*}
Let $i=\dt_1-\dt_2$. Then
\begin{equation*}
i=
\begin{cases}
\mu_2-\mu_1-2+l &\text{if $\nu_2-\nu_1$ is odd;}\\
\mu_1-\mu_2 &\text{if $\nu_2-\nu_1$ is even}.
\end{cases}
\end{equation*}
Since $\sB$ is non-semisimple, we have $l\not|\,\la_1-\la_2+1$. Thus, since $\mu\in P_1(D)$, we have $0\leq\mu_1-\mu_2\leq l-2$. It follows that $0\leq i\leq l-2$. So it remains to prove that $i\geq 2l-r-2$.
We divide four cases.

Case (1). We assume $\mu_1\leq l-2$ and $\nu_2-\nu_1$ is odd.  By \ref{block2} we have $|\sB|=2(\nu_1+\nu_2)+1$. Since $\sB$ is non-semisimple we have $\nu_1+\nu_2\geq 1$. Thus $i-2l+r+2=2\mu_2+l(\nu_1+\nu_2-1) \geq 0$.

Case (2). We assume $\mu_1\leq l-2$ and $\nu_2-\nu_1$ is even.  Since $|\sB|=2(\nu_1+\nu_2)+1\geq 3$ by \ref{block2} and $\nu_2-\nu_1$ is even, we have $\nu_1+\nu_2\geq 2$. It follows that $i-2l+r+2=2\mu_1+2+l(\nu_1+\nu_2-2)\geq 0$.

Case (3). We assume $\mu_1\geq l-1$ and $\nu_2-\nu_1$ is odd. Since $|\sB|=2(\nu_1+\nu_2+1)+1\geq 3$ by \ref{block2} and $\nu_2-\nu_1$ is odd, we have $\nu_1+\nu_2\geq 1$.  It follows that $i-2l+r+2=2\mu_2+l(\nu_1+\nu_2-1) \geq 0$.

Case (4). We assume $\mu_1\geq l-1$ and $\nu_2-\nu_1$ is even.  Since $|\sB|=2(\nu_1+\nu_2+1)+1\geq 3$ by \ref{block2} we have $\nu_1+\nu_2\geq 0$. It follows that
$i-2l+r+2=2\mu_1+2+l(\nu_1+\nu_2-2)\geq 2(\mu_1+1-l)\geq 0$.
\end{proof}

Using the description of blocks for $s_k(2,r)$,  we can determine the block for $\tuk(2,r)$ as follows.
\begin{Prop}\label{blocks for u(2,r)}
$(1)$ The algebra $\tuk(2,r)$ is semisimple for $l>r$.

$(2)$ The blocks $\sB_1^r(\lb)$ $(\lb\in\ol{\St}\cup\ol{\sI})$ are all semisimple blocks of $\tuk(2,r)$.

$(3)$ Assume $l\leq r$. Then the blocks $\sB_1^r(\lb)$ $(\lb\in\ol{\sJ})$ are all non-semisimple blocks of $\tuk(2,r)$.
\end{Prop}
\begin{proof}
If $l>r$, then by \cite[8.2]{DFW1} and \cite{EN}, we have $\tuk(2,r)=S_k(2,r)$ is semisimple. Furthermore, by \cite[8.4]{DFW2},
we have
\begin{equation}\label{relation of blocks}
\ol{\h\sB_1^r(\la)}=\sB_1^r(\lb)
\end{equation}
for $\la\in\Ga_1^r(D)$. Thus the assertions (2) and (3) follow  from \ref{semisimple non-Steinberg block}, \ref{non-semisimple block}.
\end{proof}

Now we will study the relation between indecomposable projective modules for $s_k(2,r)$ and  $\tuk(2,r)$. For $\la\in\Ga_1^r(D)$, let $\h P_1(\la)$ (resp. $\h I_1(\la)$) be the projective cover (resp. injective hull) of $\h L_1(\la)$ as a $\sk(2,r)$-module. Similarly, for $\la\in\Ga_1^r(D)$, let
$P_{1}(\lb)$ (resp. $I_1(\lb)$) be the projective cover (resp. injective hull) of $ L_1(\lb)$ as an $\tuk(2,r)$-module.

\begin{Lem}[{\cite[3.6]{EF}}]\label{block for sq(2,r)1}
Assume $\hsB_1^r(\la)=\{\la^{(0)},\cdots,\la^{(t)}\}$, where
$s^+\cdot\la^{(i)}=\la^{(i+1)}$ and $t\geq 2$. The projective covers $\h P_{1}(\la^{(i)})$
of the simple modules $\h L_1(\la^{(i)})$ have the following structure:
\begin{eqnarray*}
\h P_{1}(\la^{(0)}): \h L_1(\la^{(0)}) & \h P_{1}(\la^{(j)}):\qquad\qquad \h L_1(\la^{(j)}) & \h P_{1}(\la^{(t)}):\h L_1(\la^{(t)})    \\
   \h L_1(\la^{(1)}) &\qquad (j=1, \cdots t-1)\quad  \h L_1(\la^{(j-1)}) \ \quad \h L_1(\la^{(j+1)})& \qquad\qquad \h L_1(\la^{(t-1)}).  \\
&\qquad \qquad\qquad \ \quad\quad \h L_1(\la^{(j)}) &
\end{eqnarray*}
\end{Lem}

\begin{Prop}\label{restriction of projective module}
Assume $\hsB_1^r(\la)=\{\la^{(0)},\cdots,\la^{(t)}\}$, where
$s^+\cdot\la^{(i)}=\la^{(i+1)}$ and $t\geq 2$. Then  $\h P_{1}(\la^{(i)})\cong  P_{1}(\ol{\la^{(i)}})$ as an ${\tuk(2,r)}$-module for $1\leqslant i\leqslant t-1$.
\end{Prop}
\begin{proof}
 Let $\h Q_1(\la)$ (resp. $Q_1(\la)$) be the injective
hull of $\h L_1(\la)$ (resp. $L_1(\la)$) as a $G_1T$-module (resp. $G_1$-module).
Assume $1\leq i\leq t-1$. Then $\h I_1(\la^{(i)})\cong \h Q_1(\la^{(i)})$ by \cite[3.5]{EF}. Furthermore by \cite[3.2(10)]{Donbk} we have $\h Q_1(\la^{(i)})|_{G_1}\cong Q_1(\la^{(i)})$. It follows that $\h I_1(\la^{(i)})$ is injective as a $G_1$-module and hence $\h I_1(\la^{(i)})$ is injective as a $\tuk(2,r)$-module. Consequently, by \ref{socle-indc}, $\h I_1(\la^{(i)})|_{\tuk(2,r)}\cong I_1(\la^{(i)})$.
\end{proof}
\begin{Rem}
Note that in general $\h P_1(\la^{(i)})|_{\tuk(2,r)}$ is not isomorphic to $P_1(\ol{\la^{(i)}})$ for $i=0,t$.
\end{Rem}

\begin{Coro}\label{projective cover}
Assume $l\leq r$ and $\la\in\sJ$. Then we have $\sB_1^r(\lb)=\{\lb,\ol{\mu}\}$, where $\mu=s^-\cdot\la$. Furthermore if $|\hsB_1^r(\la)|\geq 5$, then the projective covers
of the simple modules $L_1(\lb)$ and $L_1(\ol{\mu})$ have the following structure:
\begin{eqnarray*}
P_{1}(\lb): \quad  L_1(\ol{\la }) &   P_{1}(\ol{\mu}): &\   L_1(\ol{\mu})   \\
    2 L_1(\ol{\mu}) & & 2  L_1(\ol{\la})     \\
  L_1(\ol{\la})& &\  L_1(\ol{\mu})
\end{eqnarray*}
\end{Coro}
\begin{proof}
By \ref{blocks for u(2,r)}(3), we have $\sB_1^r(\lb)=\{\ol{\la},\ol{\mu}\}$. If $|\hsB_1^r(\la)|\geq 5$, then by \ref{restriction of projective module}, we have $\h P_{1}(\la)|_{\tuk(2,r)}\cong  P_{1}(\ol{\la})$ and
$\h P_{1}(\mu)|_{\tuk(2,r)}\cong  P_{1}(\ol{\mu})$. Now using \ref{socle-indc} and \ref{restriction of projective module} we get the structure of $P_{1}(\lb)$ and $P_{1}(\mb)$.
\end{proof}

Using \ref{basic algebra of sl2} and \ref{projective cover} we can determine the basic algebra of the block $\sB_1^r(\lb)$ in the case of $|\hsB_1^r(\la)|\geq 5$.
\begin{Coro}\label{basic algebra}
Assume $\la\in\sJ$ and $|\hsB_1^r(\la)|\geq 5$. Then
the basic algebra $B$ of the block $\sB_1^r(\lb)$ is isomorphic to the algebra $\mathscr A$ defined in \S3.
\end{Coro}
\begin{proof}
Let $\mu=s^-\cdot\la$ and $i=\la_1-\la_2$. Since $\la\in\sJ$ we have $0\leq i\leq l-2$. Thus $\mu=(\la_2-1,\la_1+1)$. By \eqref{image of sl2} we conclude that $L_1(\la)$ and $L_1(\mu)$ are irreducible $u_k'(2)$-modules, and hence we have $L_1(\lb)\cong L_{i+1}$ and $L_1(\mb)\cong L_{l-i-1}$. Let $\pi_\lb$ be the natural  homomorphism from $P_1(\lb)$ to $L_1(\lb)$ and $\pi_\mb$ be the natural homomorphism from $P_1(\mb)$ to $L_1(\mb)$. Since $P_{i+1}$ and $P_{l-i-1}$ are projective $u_k'(2)$-modules, there exist $u_k'(2)$-module homomorphisms $g_{\lb}:P_{i+1}\ra P_1(\lb)$ and $g_{\mb}:P_{l-i-1}\ra P_1(\mb)$ such that $\pi_{\lb}\circ g_\lb=L_1(\lb)$ and $\pi_{\mb}\circ g_\mb=L_1(\mb)$. By \ref{projective cover} we conclude that $g_\lb$ and $g_\mb$ are all surjective. Furthermore by \ref{Pj} and \ref{projective cover}, we have  $\dim_kP_{i+1}=\dim_kP(\lb)$ and $\dim_kP_{l-i-1}=\dim_kP(\mb)$. Thus we have $P_{i+1}\cong P_1(\lb)$ and $P_{l-i-1}\cong P_1(\mb)$. Consequently, by \eqref{image of sl2} and \ref{basic algebra of sl2} we have
$B=\End_{u_k(2,r)}(P_1(\lb)\oplus P_1(\mb))\cong\End_{u_k'(2)}(P_{i+1}\oplus P_{l-i-1})\cong\mathscr A$.
\end{proof}

\section{The basic algebra of $\End_{u_k(2)}(\Ogk^{\otimes r})$}
Using the result of \S4, we will determine the basic algebra of $\End_{u_k(2)}(\Ogk^{\otimes r})$ in this section.
By \cite[8.2]{DFW1},  $\tuk(2,r)=S_k(2,r)$ is semisimple for $l> r$. Hence
if $l> r$, then $\End_{u_k(2)}(\Ogk^{\otimes r})=\End_{\tuk(2,r)}(\Ogk^{\otimes r})$ is semisimple and it is the quotient of Hecke algebra modulo the kernel of the action on the tensor space. So we assume $l\leq r$ from now on.

Write  $\Ogk^{\otimes r}=\oplus_{1\leq i\leq h}m_iE_i$ with $m_i\geq 1$, where the $E_i$ are indecomposable $\tuk(2,r)$-modules and pairwise non-isomorphic. Let $N=\oplus_{1\leq i\leq h}E_i$.  Then  $$\La_r:=\End_{u_k(2)}(N)=\End_{\tuk(2,r)}(N)$$ is the basic algebra of
$\End_{u_k(2)}(\Ogk^{\ot r})$. Let $B$ be a block of $\tuk(2,r)$. We will denote $N(B)=\oplus_{E_i\in B}E_i$, where $E_i\in B$ means that every composition factor of $E_i$ belongs to $B$. Then we have $N=\oplus_BN(B)$, where $B$ runs through all blocks of  $\tuk(2,r)$. Furthermore, we let $\La_r(B)=\End_{\tuk(2,r)}(N(B))$.
Then
$$\La_r=\prod_B \La_r(B).$$
By \ref{blocks for u(2,r)} we have
\begin{equation}\label{deco}
\La_r\cong k^a\times \prod_{\lb\in\ol{\sJ}}\La_r(\sB_1^r(\lb)),
\end{equation}
where $a=\#\{B\mid B\text{ is the semisimple block of } \tuk(2,r),\,N(B)\not=0\}$.
Now we will first calculate the number $a$ and $|\ol{\sJ}|$.
Clearly, we have
$$
\sJ=
\begin{cases}
\{(\frac{r+i}{2},\frac{r-i}{2})\mid 0\leq i\leq l-2,\,r-i\text{ is  even}\}&\text{if $r\geq 2l-2$,}\\
\{(\frac{r+i}{2},\frac{r-i}{2})\mid 2l-r-2\leq i\leq l-2,\,r-i\text{ is even}\}&\text{if $r< 2l-2$.}
\end{cases}
$$
Thus we immediately get the following results.
\begin{Lem}
$(1)$ if $r\geq 2l-2$, then $|\ol{\sJ}|=\frac{l-1}{2}$;

$(2)$ If $r<2l-2$, then
$$
|\ol{\sJ}|=
\begin{cases}
\frac{r-l+1}{2}&\text{if $r$ is even,}\\
\frac{r-l+2}{2}&\text{if $r$ is odd.}
\end{cases}
$$
\end{Lem}

\begin{Lem}\label{iso types}
The isomorphism
types of the indecomposable summands of $\Ogk^{\otimes r}$ as an $\tuk(2,r)$-module are
precisely

(i) the simple modules $L_1(\lb)$ for $\la\in\sI\cup\sJ;$

(ii) the simple modules $L_1(\lb)$ for $\lb\in\ol{\mathrm{St}};$

(iii) the modules  $P_1(\lb)$ and $P_1(\mb)$, where $\la\in\sJ$, $\mu=s^+\cdot\la$ and $|\h\sB_1^r(\la)|\geq 5;$

(iv) the module $P_1(\lb)$, where $\la\in\sJ$ and $|\h\sB_1^r(\la)|= 3$.
\end{Lem}
\begin{proof}
By \cite[\S6]{EF}
the isomorphism
types of the indecomposable summands of $\Ogk^{\otimes r}$ as an $s_k(2,r)$-module are
precisely
\begin{itemize}
\item[(i)]
the simple modules $\h L_1(\la)$ for
$\la\in\sI\cup\sJ;$
\item[(ii)] the simple modules $\h L_1(\la)$ for $\la\in\St;$
\item[(iii)]
the modules  $\h P_1(s^m\cdot\la)$ for $\la\in\sJ$ and $|m|\leq\frac{1}{2}(|\h\sB_1^r(\la)|-3)$.
\end{itemize}
Thus the assertion follows from
\ref{restriction of projective module}.
\end{proof}

\begin{Coro}
$(1)$ If $r\geq 2l-2$, then $a=1$.

$(2)$ If $l\leq r<2l-2$ and $r$ is even, then $a=l-\frac{r}{2}$;

$(3)$ If $l\leq r<2l-2$ and $r$ is odd, then $a=l-\frac{r+1}{2}$.
\end{Coro}
\begin{proof}
By \ref{blocks for u(2,r)} and \ref{iso types}, we have $a=|\ol\St|+|\ol{\sI}|$. By \ref{St} we have $|\ol\St|=1$. If $r\geq 2l-2$, then ${\sI}=\emptyset$ and hence $a=1$. If $l\leq r<2l-2$ and $r$ is even, then $ {\sI}=\{(\frac{r+i}{2},\frac{r-i}{2})\mid i=0,2,4,\cdots,2l-r-4\}$ and hence $a=l-\frac{r}{2}$. If $l\leq r<2l-2$ and $r$ is odd, then $ {\sI}=\{(\frac{r+i}{2},\frac{r-i}{2})\mid i=1,3,5,\cdots,2l-r-4\}$ and hence $a=l-\frac{r+1}{2}$.
\end{proof}

Finally, we will determine $\La_r(\sB_1^r(\lb))$ for $\lb\in\ol\sJ$.
\begin{Lem}\label{the number of simple modules in a block}
Assume $\la\in\sJ$. Write $\la=\mu+l\nu$, where $\mu\in P_1(D)$ and $\nu\in P(D)$. If $\mu_1\leq l-2$, then $|\h\sB_1^r(\la)|\geq 5$.
\end{Lem}
\begin{proof}
Write $\la_2=sl+t$, where $s\geq 0$ and $0\leq t<l$. Let $i=\la_1-\la_2$. Then $\la=(i+t,t)+l(s,s)$. Since $\la\in\sJ$, we have $(i+t,t)\in P_1(D)$ and hence $\mu=(i+t,t)$ and $\nu=(s,s)$. Now we assume $\mu_1\leq l-2$. Since $\la\in\sJ$, we have $i\geq 2l-r-2$. Thus, since $t=\la_2-sl=\frac{r-i}{2}-sl$, we have
$$0\leq l-2-\mu_1=l-2-(i+t)=(s+1)l-2-\frac{r+i}{2}\leq (s+1)l-2-l+1=sl-1.$$
It follows $s\geq 1$. Thus by \ref{block2}, we have $|\h\sB_1^r(\la)|=4s+1\geq 5$.
\end{proof}

\begin{Coro}\label{equivalent condition}
Assume $\la\in\sJ$. Then $|\h\sB_1^r(\la)|=3$ if and only if $\la\in P_1(D)$.
\end{Coro}
\begin{proof}
Let $i=\la_1-\la_2$ and write $\la_2=sl+t$ with $s\geq 0$ and $0\leq t<l$.  Then $\la=\mu+l\nu$, where $\mu=(i+t,t)$ and $\nu=(s,s)$. Since $\la\in\sJ$, we have $\mu\in P_1(D)$.
If $|\h\sB_1^r(\la)|=3$ then by \ref{the number of simple modules in a block} we have $\mu_1\geq l-1$. Thus by \ref{block2} we have $|\h\sB_1^r(\la)|=4s+3$. It follows $s=0$ and hence $\la=\mu\in P_1(D)$.

If $\la\in P_1(D)$, then $\la=\mu$ and $\nu=0$. Let $i=\la_1-\la_2$. Since $\la\in\sJ$, we have $\mu_1=\la_1=\frac{r+i}{2}\geq l-1$. Thus by \ref{block2}, we have $|\h\sB_1^r(\la)|=3$.
\end{proof}



\begin{Thm}
$(1)$ If $\la\in\sJ\cap P_1(D)$, then $\La_r(\sB_1^r(\lb))$ has quiver as shown in Figure 2, and is given by
the relations $\bt\al=0$;
\begin{figure}[h]
\begin{center}
\setlength{\unitlength}{.5cm}
\begin{picture}(18,2.9)
 \put(6.9,1){$\bullet$} \put(6.7,0.3){$X$}
\put(10.9,1){$\bullet$} \put(10.7,0.3){$Y$}
  \put(7.4,1.3){\vector(1,0){3.2}}
\put(9,1.5){$\alpha$}
\put(10.6,0.9){\vector(-1,0){3.2}} \put(9,0.2){$\beta$}
\end{picture}
\caption{}
\end{center}
\end{figure}

$(2)$ If $\la\in\sJ$ and $\la\not\in P_1(D)$, then $\La_r(\sB_1^r(\lb))$ has quiver as shown in Figure 3, and is given by
the relations $\al_i\bt_j=0$, $\al_1\bt_1=\al_2\bt_2$, $\bt_i\al_j=0$,  $\ga\dt=0$, $\ga\bt_i=0$, $\al_i\dt=0$ and $\bt_1\al_1=\bt_2\al_2=\dt\ga$ for $i\not=j$.

\begin{figure}[ht]\label{Afivefiguretwo}
\begin{center}
\setlength{\unitlength}{.5cm}
\begin{picture}(12,3.1)
\put(0.4,1.6){$\bullet$} \put(0.4,0.8){$X$}
\put(5.5,1.6){$\bullet$} \put(5.5,0.8){$Y$}
\put(10,1.6){$\bullet$} \put(9.9,0.8){$Z$}
\put(1.4,2.7){\vector(1,0){3.5}} \put(3,2.9){$\alpha_{1}$}
\put(1.4,1.9){\vector(1,0){3.5}} \put(3,2.2){$\alpha_{2}$}
\put(4.8,1.55){\vector(-1,0){3.5}} \put(3,0.9){$\beta_{1}$}
\put(4.8,0.7){\vector(-1,0){3.5}} \put(3,0){$\beta_{2}$}
\put(6.2,2.1){\vector(1,0){3.2}} \put(7.4,2.3){$\delta$}
\put(9.4,1.5){\vector(-1,0){3.2}} \put(7.4,0.8){$\ga$}
\end{picture}
\caption{}
\end{center}
\end{figure}
\end{Thm}
\begin{proof}
Let $X_0=L_1(\mb)$ and $X_1=L_1(\lb)$, where $\mu=s^-\cdot\la$.
Let $P(X_0)=P_1(\mb)$ and $P(X_1)=P_1(\lb)$. If $\la\in\sJ\cap P_1(D)$, then by \ref{iso types} and \ref{equivalent condition} we have $N(\sB_1^r(\lb))=X_1\oplus P(X_1)$, and hence $\La_r(\sB_1^r(\lb))=\End_{\tuk(2,r)}(X_1\oplus P(X_1))$. We write $\La_r(\sB_1^r(\lb))$ as
$$\left(\begin{matrix}\phi_{11}& \phi_{12}\cr \phi_{21} & \phi_{22}\end{matrix}
\right)$$ where $\phi_{11} \in {\rm End}(X_1)$, $\phi_{22}
\in {\rm End}(P(X_1))$, $\phi_{12}\in{\rm
Hom}(X_{1}, P(X_1))$ and $\phi_{21} \in {\rm Hom}(P(X_1), X_{1})$. Then
$$\La_r(\sB_1^r(\lb))\cong \left(\begin{matrix} k & k \cr k& {\rm End}(P(X_1))\end{matrix}\right).
$$
Let $\al$ be the natural surjective $\tuk(2,r)$-module homomorphism from $P(X_1)$ to $X_1$ and let $\bt$ be the natural injective $\tuk(2,r)$-module homomorphism from $X_1$ to $P(X_1)$. Then $\bt\al=0$. From this, one can easily conclude (1).

Now we assume $\la\in\sJ$ and $\la\not\in P_1(D)$. Then by
\ref{iso types} and \ref{equivalent condition}, we have $N(\sB_1^r(\lb))=X_1\oplus P(X_0)\oplus P(X_1)$, and hence
we have $\La_r( \sB_1^r(\lb))=\End_{\tuk(2,r)}(X_1\oplus P(X_0)\oplus P(X_1))$. We write $\La_r( \sB_1^r(\lb))$ as
$$\left(\begin{matrix}\phi_{11}& \phi_{12}\cr \phi_{21} & \phi_{22}\end{matrix}
\right)$$ where $\phi_{11} \in {\rm End}(X_1)$, $\phi_{22}
\in {\rm End}(P(X_0)\oplus P(X_1))$, $\phi_{12}\in{\rm
Hom}(X_1, P(X_0)\oplus P(X_1))$ and $\phi_{21} \in {\rm Hom}(P(X_0)\oplus P(X_1), X_{1})$. Then
$$\La_r( \sB_1^r(\lb))\cong \left(\begin{matrix} k & k \cr k& B\end{matrix}\right),
$$
where $ B={\rm End}(P(X_0)\oplus P(X_1))$ is the basic algebra of $\sB_1^r(\lb)$.
Thus (2) follows from \ref{basic algebra}.
\end{proof}

\section*{Acknowledgement}
We would like to thank professor Susumu Ariki for his suggestion on this problem. We also thank
professor Bang Ming Deng for some useful
discussions and thank the referee for some useful comments.

\end{document}